\documentclass[11pt,twoside]{amsart}

\usepackage[latin1]  {inputenc}%
\usepackage[T1]      {fontenc }%
\usepackage          {amsmath }%
\usepackage          {amsfonts}%
\usepackage          {amssymb }%
\usepackage          {amsthm  }%
\usepackage          {a4wide  }%
\usepackage          {url     }%
\usepackage          {tikz    }%
\usepackage[bookmarks=false,pdfborder={0 0 0.05}]{hyperref}
\usepackage[all]{xy}
\usepackage{tikz-cd}

\usepackage{enumerate, amsmath, amsfonts, amssymb, amsthm,  wasysym, graphics, graphicx, xcolor, frcursive,comment,bbm}

\usepackage{etex}

\definecolor{darkblue}{rgb}{0.0,0,0.7} 
\definecolor{darkred}{rgb}{0.7,0,0} 

\usepackage{hyperref}
\usepackage[all]{xy}
\usepackage[T1]{fontenc}

\usepackage{MnSymbol}

\def\defn#1{{\sf #1}}

\newcommand{\edge}{\mathbin{\tikz [semithick, baseline=-0.2ex,-latex, ->] \draw [-] (0pt,0.4ex) -- (1em,0.4ex);}} 

\newcommand{\RR}{\mathbb R}
\newcommand{\ZZ}{\mathbb Z}

\newcommand{\spanz}{\mathrm{span}_{\mathbb{Z}}}

\DeclareMathOperator{\Red}{Red}

\DeclareMathOperator{\GL}{GL}

\DeclareMathOperator{\B}{\mathcal{B}}

\DeclareMathOperator{\NN}{\mathbb{N}}

\newtheorem{theorem}{Theorem}[section]
\newtheorem{corollary}[theorem]{Corollary}
\newtheorem{proposition}[theorem]{Proposition}
\newtheorem{lemma}[theorem]{Lemma}

\theoremstyle{definition}
\newtheorem{definition}[theorem]{Definition}

\newtheorem{remark}[theorem]{Remark}
\newtheorem{example}[theorem]{Example}

\newtheorem{question}[theorem]{Question}

\title[Reflection factorizations and quasi-Coxeter elements]{Reflection factorizations and quasi-Coxeter elements}

\author[P.~Wegener]{Patrick Wegener}
\address{Patrick Wegener, Leibniz Universit\"at Hannover, Germany}
\email{patrick.wegener@math.uni-hannover.de}

\author[S.~Yahiatene]{Sophiane Yahiatene}
\address{Sophiane Yahiatene, Universit\"at Bielefeld, Germany}
\email{syahiate@math.uni-bielefeld.de}

\begin{document}

\begin{abstract}
We investigate the so-called \textit{dual Matsumoto property} or \textit{Hurwitz action} in finite, affine and arbitrary Coxeter groups. In particular, we want to investigate how to reduce reflection factorizations and how two reflection factorizations of the same element are related to each other. We are motivated by the dual approach to Coxeter groups proposed by Bessis in \cite{Bessis} and the question whether there is an anlogue of the well known Matsumoto property for reflection factorizations. Our aim is a substantial understanding of the Hurwitz action. We therefore reprove uniformly results of Lewis--Reiner as well as Baumeister--Gobet--Roberts and the first author on the Hurwitz in finite Coxeter groups. Further we show that in an arbitrary Coxeter group all reduced reflection factorizations of the same element appear in the same Hurwitz orbit after a suitable extension by simple reflections.

As parabolic quasi-Coxeter elements play an outstanding role in the study of the Hurwitz action, we aim to characterize these elements. We give characterizations of maximal parabolic quasi-Coxeter elements in arbitrary Coxeter groups as well as a characterization of all parabolic quasi-Coxeter elements in affine Coxeter groups.
\end{abstract}

\maketitle


\section{Introduction}

The so-called \textit{Matsumoto property} states that for a Coxeter system $(W,S)$ any two $S$-reduced factorizations of the same element can be transformed one into the other by just using the braid relations (see \cite{Matsumoto}). In the \textit{dual approach} to Coxeter groups, as suggested by Bessis in \cite{Bessis}, the generating set $S$ is replaced by the set $T$ of all reflections for $(W,S)$. It naturally arises the question whether there is an analogue of the Matsumoto property for the dual approach. Namely, given two $T$-reduced factorizations of the same element, is there a procedure to transform both factorizations one into the other? Given an element $w \in W$ and a factorization $w=t_1\cdots t_m$ into reflections (reduced or not reduced), one may transform this factorization as follows to obtain new factorizations for $w$:
\begin{align*}
(t_1,\dots, t_m) & \sim(t_1,\dots, t_{i-1}, t_i t_{i+1}t_i,t_i,t_{i+2},\dots, t_m),\\
(t_1,\dots, t_m) & \sim(t_1,\dots, t_{i-1}, t_{i+1}, t_{i+1}t_it_{i+1},t_{i+2},\dots, t_m).
\end{align*}
These transformations are called \textit{Hurwitz moves}. In fact, the Hurwitz moves extend to an action of the \textit{$m$-strands braid group}, called \textit{Hurwitz action} (see Section \ref{sec:Hurwitz} for the precise definitions). It has been shown by Lewis and Reiner \cite[Corollary 1.4]{LR16} that the Hurwitz action in finite Coxeter groups can be used to reduce reflection factorizations. Their proof is case-based, including large computer calculations for the exceptional types. In an attempt to better understand the Hurwitz action as well as the dual approach, as our first main result we provide a uniform proof of the following result.

\begin{theorem}[{Lewis--Reiner}] \label{thm:main1}
Let $(W,S)$ be a finite Coxeter system and $w=t_{1}\cdots t_{m+2k}\in W$ a reflection factorization with $\ell_{T}(w)=m$ and $k\in \ZZ_{\geq 0}$. Then there exists a braid $\sigma \in \B_{m+2k}$ such that
\[\sigma(t_{1},\ldots, t_{m+2k})=(r_{1},\ldots,r_{m},r_{i_{1}},r_{i_{1}},\ldots,r_{i_{k}},r_{i_{k}}).\]
\end{theorem}

The previous theorem fails to be true in arbitrary Coxeter groups. Nevertheless, we show that in an arbitrary Coxeter group all reduced reflection factorizations of the same element appear in the same orbit with respect to the Hurwitz action. We therefore have to extend these reduced reflection factorizations by suitable simple reflections.

\begin{theorem} \label{thm:main1.1}
Let $(W,S)$ be a Coxeter group,  $w=s_{1}\cdots s_{m}$ a reduced factorization of $w\in W$ in simple reflections and $(t_{1},\ldots,t_{n})$ a reduced reflection factorization of $w$. Then there exist $q_{1},\ldots,q_{m-n}\in S$  and a braid $\sigma \in \B_m$ such that  
$$
\sigma(s_{1},\ldots,s_{m})=(q_{1},\ldots,q_{m-n},t_{1},\ldots,t_{n}).$$
\end{theorem}

Besides the fact that the Hurwitz action preserves the product of the transformed tuple, there is another natural invariant of the Hurwitz action. Namely, the Hurwitz action preserves the group generated by the corresponding tuples. Therefore quasi-Coxeter elements play an important role in the study of the Hurwitz action. An element $w$ is called a \textit{quasi-Coxeter element} (resp. \textit{parabolic quasi-Coxeter element}) if there exists a $T$-reduced factorization of $w$ which generates the group $W$ (resp. a parabolic subgroup of $W$). We call an element a \textit{proper} parabolic quasi-Coxeter element if it is a quasi-Coxeter element in a proper parabolic subgroup. As our next main result we provide a characterization of ``maximal'' parabolic quasi-Coxeter elements for arbitrary Coxeter groups of finite rank by means of the \textit{absolute order} $\leq_T$ (see Definition \ref{def:AbsOrder}) and the \textit{parabolic closure} (see Section \ref{sec:parabolic}).

\begin{theorem}\label{prop:prefix_quasi_cox}
Let $(W,S)$ be a Coxeter system of rank $n\in \NN$ and $x \in W$ with $\ell_T(x)= n-1$. The element $x$ is a proper parabolic quasi-Coxeter element if and only if there exists a quasi-Coxeter element $w \in W$ with $x \leq_T w$ and $P(x)\neq W$. In this case, the parabolic closure $P(x)$ of $x$ has rank $n-1$.
\end{theorem}

The importance of parabolic quasi-Coxeter elements in the study of the Hurwitz action is emphasized by a result of Baumeister, Gobet, Roberts and the first author \cite[Theorem 1.1]{BGRW17}. They showed that for finite Coxseter groups the parabolic quasi-Coxeter elements are precisely those elements with just one orbit for the Hurwitz action. Their proof is case-based. We are able to give a uniform proof, if only for Weyl groups.

\begin{theorem}\label{thm:main2}
Let $W$ be a Weyl group and $w \in W$ be a parabolic quasi-Coxeter element. Then the Hurwitz action is transitive on the set of reduced reflection factorizations of $w$.
\end{theorem}

Using the previous results, we are able to characterize parabolic quasi-Coxeter elements in affine Coxeter groups. Note that an element of finite order in an affine Coxeter group is also called \textit{elliptic}.

\begin{corollary} \label{main4}
Let $(W,S)$ be an irreducible affine Coxeter system with set of reflections $T$ and let $x \in W$. Then $x$ is a proper parabolic quasi-Coxeter element if and only if there exists a quasi-Coxeter element $w\in W$ such that $x \leq_T w$ and $P(x) \neq W$.

Equivalently, $x$ is a proper parabolic quasi-Coxeter element if and only if there exists a quasi-Coxeter element $w\in W$ such that $x \leq_T w$ and $x$ is elliptic.
\end{corollary}

Note that this result and Theorem \ref{prop:prefix_quasi_cox} partially generalize results of Bessis \cite[Lemma 1.4.3]{Bessis} as well as Paolini and Salvetti \cite[Theorem 3.22 (i)]{PS19} on parabolic Coxeter elements in finite and affine Coxeter groups.

\section{Background}

\subsection{Generalities on Coxeter groups}
Recall that a \defn{Coxeter group} is a group $W$ given by a presentation
$$
W = \langle  S \mid (st)^{m_{st}} = 1 ~\forall s,t \in S \rangle,
$$
where $(m_{st})_{s,t \in S}$ is a symmetric matrix with entries in $\ZZ_{\geq 1} \cup \{ \infty \}$. These entries have to satisfy $m_{ss}=1$ for all $s \in S$ and $m_{st} \geq 2$ for all $s \neq t$ in $S$. If $m_{st} = \infty$, then there is no relation for $st$ in the above presentation. The pair $(W,S)$ is called a \defn{Coxeter system}, $S$ is called the set of \defn{simple reflections} and $|S|$ is called the \defn{rank} of $(W,S)$. Further, if $|W|$ is finite the system is called \defn{finite} and otherwise it is called \defn{inifinite}. We assume all Coxeter systems in this paper to be of finite rank.

To each Coxeter system $(W,S)$ there is an associated labeled graph, called \defn{Coxeter diagram} and denoted by $\Gamma(W,S)$. Its vertex set is given by $S$ and there is an edge between distinct $s, t \in S$ labeled by $m_{st}$ if $m_{st} > 2$. The Coxeter system $(W,S)$ is called \defn{irreducible} if $\Gamma(W,S)$ is connected. 

Each $w \in W$ can be written as a product $w = s_{1} \cdots s_{k}$ with $s_{i} \in S$. The \defn{length} $\ell(w) = \ell_S(w)$ is defined to be the smallest integer $k$ for which such an expression exists. The expression $w = s_{1} \cdots s_{k}$ is called \defn{($S$-)reduced} if $k = \ell(w)$.

\medskip
Let $(W,S)$ be a Coxeter system and let $V$ be a vector space over $\mathbb{R}$ with basis $\Delta = \{e_s \, \mid \, s \in S\}$. We equip $V$ with a symmetric bilinear form $B$ by setting
\[ B(e_{s}, e_{t})= - \cos \frac{\pi}{m_{st}}\]
for all $s, t \in S$. This term is understood to be $-1$ if $m_{st}= \infty$. The group $W$ can be embedded into $\operatorname{GL}(V)$ via its natural representation (or Tits representation) $\sigma: W \to \GL(V)$ that sends $s \in S$ to the reflection
\[ \sigma_s: V \rightarrow V, ~v \mapsto v- 2 B(e_s, v) e_s. \]
We set $w(e_s) := \sigma(w)(e_s)$ and
$$
\Phi = \Phi(W,S):= \{ w(e_s) \mid w \in W,~s \in S \}.
$$
The set $\Phi$ is called the \defn{root system} for $(W,S)$ and we refer to $\Delta$ as the \defn{simple system} for $\Phi$. We call a root $\alpha= \sum_{s \in S} a_s e_s$
\defn{positive} and write $\alpha>0$ if $a_s \geq 0$ for all $s \in S$ and \defn{negative} if $a_s \leq 0$ for all $s \in S$. Let $\Phi^+$ be the set consisting of the positive roots. It turns out that
$\Phi$ decomposes into positive and negative roots, that is, $\Phi= \Phi^+ \dot{\cup} -\Phi^+$. 

If $\alpha = w(e_s) \in \Phi$ for some $w \in W$ and $s \in S$, then $wsw^{-1}$ acts as a reflection on $V$. It sends $\alpha$ to $-\alpha$ and fixes pointwise the hyperplane orthogonal to $\alpha$. We set $s_{\alpha} = wsw^{-1}$ and call $T= \{ wsw^{-1}\mid w \in W, ~s \in S\}$ the set of \defn{reflections} for $(W,S)$.

\subsection{Reflection factorizations and the Hurwitz action} \label{sec:Hurwitz}
Since $S \subseteq T$, we have $W = \langle T \rangle$. Therefore each $w \in W$ can be written as a product $w = t_{1} \cdots t_{m}$ with $t_{i} \in T$. We call this a \defn{reflection factorization} for $w$. The \defn{reflection length} $\ell_T(w)$ is defined to be the smallest integer $m$ for which such a factorization exists. The factorization $w = t_{1} \cdots t_{m}$ is called \defn{($T$-)reduced} or \defn{reduced reflection factorization} if $m = \ell_T(w)$. For $w \in W$ with $m= \ell_T(w)$ we further define its \defn{set of reduced reflection factorizations} as
$$
\Red_T(w):= \{ (t_1, \ldots, t_m) \in T^m \mid w = t_1 \cdots t_m \}.
$$
Note that we will use the terminology of a reflection factorization for an element $w \in W$ synonymously for the product $w=t_1\cdots t_m$ of reflections as well as the tuple $(t_1,\ldots, t_m)$. There is a nice criterion for finite Coxeter groups to determine whether a reflection factorization is reduced.

\begin{lemma}[{Carter, \cite[Lemma 3]{Ca70}}] \label{lem:Carter}
Let $(W,S)$ be a finite Coxeter system with root system $\Phi$. The reflection factorization $s_{\alpha_1} \cdots s_{\alpha_m}$ $(\alpha_i \in \Phi)$ is reduced if and only if $\alpha_1, \ldots , \alpha_m$ are linearly independent.
\end{lemma}

For an element $w \in W$ with $\ell_T(w)=m$, there is a natural action of the \defn{braid group} on $m$ strands on the set $\Red_T(w)$. More precisely, the braid group $\B_m$ is the group given by the following presentation
$$
\B_m = \langle \sigma_1, \ldots , \sigma_{m-1} \mid \sigma_{i}\sigma_{i+1}\sigma_{i}= \sigma_{i+1}\sigma_{i}\sigma_{i+1}, ~\sigma_{i}\sigma_{j}=\sigma_{j}\sigma_{i} ~\text{for }|i-j|>1 \rangle.
$$
This group acts on the set $\Red_T(w)$ in terms of its generators (and their inverse elements) as follows
\begin{align*}
\sigma_i (t_1,\dots, t_m) & =(t_1,\dots, t_{i-1}, t_i t_{i+1}t_i,t_i,t_{i+2},\dots, t_m),\\
\sigma_i^{-1} (t_1,\dots, t_m) & =(t_1,\dots, t_{i-1}, t_{i+1}, t_{i+1}t_it_{i+1},t_{i+2},\dots, t_m).
\end{align*}
It is straightforwad to check that this indeed extends to an action of $\B_m$ on $\Red_T(w)$, called \defn{Hurwitz action}. In the same way we can consider the Hurwitz action on arbitrary (not necessarily reduced) reflection factorizations of a given fixed length. We call the action of a generator $\sigma_i$ (resp. $\sigma_i^{-1}$) on the tuple $(t_1,\dots, t_m)$ a \defn{Hurwitz move}. Further we write 
$$
(t_1,...,t_m) \sim (r_1, ..., r_m)
$$
if both tuples lie in the same orbit with respect to the Hurwitz action. In this case we also say that both tuples are \defn{Hurwitz equivalent}. An orbit with respect to the Hurwitz action is called \defn{Hurwitz orbit}. We want to emphasize the following two invariants of the Hurwitz action.
\begin{remark}
Let $(t_1,...,t_m) \sim (r_1, ..., r_m)$ be Hurwitz equivalent reflection factorizations. Then:
\begin{enumerate}
\item[(a)] $\langle t_1,...,t_m \rangle = \langle r_1,...,r_m \rangle$;
\item[(b)] the tuples $(t_1,...,t_m)$ and $(r_1, ..., r_m)$ share the same multiset of conjugacy classes.
\end{enumerate}

\end{remark}

\subsection{Parabolic subgroups} \label{sec:parabolic}
For each subset $I \subseteq S$ the subgroup $W_I = \langle I \rangle$ is called a \defn{standard parabolic subgroup} of $W$. A subgroup of the form $w W_I w^{-1}$ for some $w \in W$ and $I \subseteq S$ is called a \defn{parabolic subgroup}. Note that if $wW_I w^{-1}$ is a parabolic subgroup, then $(wW_I w^{-1}, w I w^{-1})$ is itself a Coxeter system. 

Let $X \subseteq W$ be a finite set. The \defn{parabolic closure} of $X$ is defined to be the intersection of all parabolic subgroups containing $\langle X \rangle$. We denote the parabolic closure of X by $P(X)$. By \cite[Theorem 1.2]{Qi} we have that $P(X)$ is indeed itself a parabolic subgroup. In particular, $P(X)$ is precisely the smallest (with respect to inclusion) parabolic subgroup of $W$ containing $X$. If $X=\{x_1, \ldots , x_n\}$, we also write $P(X)=P(x_1,\ldots, x_n)$.

\subsection{Weyl groups and affine Coxeter groups}
Let $V$ be a real vector space with positive definite symmetric biilinear form $\langle -,- \rangle : V \times V \rightarrow \RR$. Let $\Phi$ be a crystallographic root system in $V$ (in the sense of \cite{Hum90}) with simple system $\Delta$. The set 
$$\Phi^{\vee}:=\{ \alpha^{\vee} \mid \alpha \in \Phi\},$$
where $\alpha^{\vee}:=\frac{2 \alpha}{(\alpha \mid \alpha)}$, is again a crystallographic root system in $V$ with simple system $\Delta^{\vee}:=\{ \alpha^{\vee} \mid \alpha \in \Delta \}$. The root system $\Phi^{\vee}$ is called the \defn{dual root system} and its elements are called \defn{coroots}. For a set of roots $R \subseteq \Phi$ we put $W_R:=\langle s_{\alpha} \mid \alpha \in R \rangle$ and call $W_{\Phi}$ a \defn{Weyl group}.

For a set of vectors $\Phi \subseteq V$ we set $L(\Phi):=\spanz(\Phi)$. If $\Phi$ is a crystallographic root system, then $L(\Phi)$ is an integral lattice, called \defn{root lattice}. In the latter case we call $L(\Phi^{\vee})$ the \defn{coroot lattice}.

Let us fix a crystallographic root system $\Phi$ in $V$ with simple system $\Delta$. For each $\alpha \in \Phi$ and each $k \in \ZZ$, the set 
$ H_{\alpha, k}:= \{v \in V \mid (v \mid \alpha)=k \}$
defines an affine hyperplane. We define the \defn{affine reflection} $s_{\alpha,k}$ in $H_{\alpha, k}$ by 
$$ s_{\alpha, k}: V \rightarrow V, v \mapsto v -((v \mid \alpha) -k) \alpha^{\vee}.$$
Then $s_{\alpha, k}$ fixes $H_{\alpha, k}$ pointwise and sends $0$ to $k \alpha^{\vee}$.

The group 
$$W_{a, \Phi} := \langle s_{\alpha, k} \mid \alpha \in \Phi, k \in \ZZ \rangle $$
is called \defn{affine Weyl group} associated to $\Phi$.

By \cite[Proposition 4.2, Theorem 4.6]{Hum90} the group $W_{a, \Phi}$ is the semidirect product of the Weyl group $W_{\Phi}$ and the coroot lattice $L(\Phi^{\vee})$.
Further, the group $W_{a, \Phi}$ is a Coxeter group. If $\Phi$ is irreducible, then $(W_{a, \Phi}, S)$ is a Coxeter system, where
$$S:= \{s_{\alpha} \mid \alpha \in \Delta \} \cup \{ s_{\widetilde{\alpha},1} \},$$
and $\widetilde{\alpha}$ is the highest root of $\Phi$ with respect to $\Delta$.
The set of reflections for $(W_{a, \Phi}, S)$ is given by the set of affine reflections, that is, by the set
$
 \{ s_{\alpha, k} \mid \alpha \in \Phi, ~k \in \ZZ\}.
$
Therefore we also call the affine Weyl group $W_{a, \Phi}$ an \defn{affine Coxeter group} and the pair $(W_{a, \Phi}, S)$ is called \defn{affine Coxeter system}. Note that there is a canonical projection from $W_{a, \Phi}$ to the underlying finite Weyl group, namely
\begin{align} \label{equ:ProjAffFin}
p:W_{a, \Phi} \rightarrow W_{\Phi},~s_{\alpha,k}\mapsto s_{\alpha}.
\end{align}

\section{Reflection subgroups and the Hurwitz action}

In this section we describe the connection between the Hurwitz action and the Bruhat graph for Coxeter systems of finite rank. We benefit from results of Dyer \cite{Dye87, Dye90}.

Let $(W,S)$ be a Coxeter system of finite rank and $T$ its set of reflections. A subgroup $W'$ is called a \defn{reflection subgroup} if it is generated by the reflections it contains, that is, $W'= \langle W' \cap T \rangle$. The reflection subgroup $W'$ is by \cite[Thereom 3.3]{Dye90} itself a Coxeter group with simple system 
\[\chi(W')=\lbrace t\in W' \cap T \mid \ell_{S}(tt')>\ell_{S}(t) \text{ for all } t'\in W'\cap T \text{ with } t\neq t' \rbrace.\]
The set $\chi(W')$ is called the \defn{canonical simple system} for $W'$. For later purpose we explicitely describe how to obtain the set $\chi(W')$ provided the set $T' = W' \cap T$ is known. The result ist due to Dyer.

\begin{lemma}[{\cite[Proposition 3.7]{Dye90}}] \label{le:DyerAlg}
Let $T' \subseteq T$ be a finite set. For $i \in \NN$ define sets $T_i$ as follows. Set $T_0 = T'$. Given $T_i$, set $T_{i+1}=T_i$ if $\chi(\langle t,t' \rangle)= \{t,t'\}$ for all $t,t' \in T$. Otherwise, choose $t,t' \in T$ with $\chi(\langle t,t' \rangle)\neq  \{t,t'\}$ and set $T_{i+1}=(T_i \setminus \{t,t'\}) \cup \chi(\langle t,t' \rangle)$. Then there exists some $i \in \NN$ with $T_i = T_{i+1}$ and $\chi(\langle T' \rangle)=T_i$ for this $i$. 
\end{lemma}

By \cite[Theorem 3.3 (i)]{Dye90} the set of reflections for $(W',\chi(W'))$ is precisely 
\begin{align}
    \bigcup_{w\in W'}w \chi(W') w^{-1}= W'\cap T.\label{eq:set_of_refl}
\end{align}
In \cite[Lemma 3.2]{Dye90} it is proven that the canonical simple system of a reflection subgroup $W'$ behaves under conjugation with $s\in S$ as follows
\begin{align}
\chi(sW's)=\begin{cases}
s\chi(W')s&,~ s\notin \chi(W')\\
\chi(W')&,~ \text{else}.
\end{cases} \label{eq:a1}
\end{align}
The \defn{Bruhat graph} $\Omega_{(W,S)}$ attached to the Coxeter system $(W,S)$ is a directed graph with vertex set $W$ and there is a directed edge from $x$ to $y$ if there exists $t\in T$ such that $y=xt$ and $\ell_{S}(x)<\ell_{S}(y)$. The full subgraph consisting of vertices $V\subseteq W$ is denoted by $\Omega_{(W,S)}(V)$. Moreover, by \cite[Theorem 1.4]{Dye91} for any reflection subgroup $W'\subseteq W$ and $w\in W$ there are isomorphisms of directed graphs
\begin{align}
\Omega_{(W,S)}(wW')\cong \Omega_{(W,S)}(W') \cong \Omega_{(W',\chi(W'))}. \label{eq:a2}
\end{align}

\medskip
As a consequence of the definition, for each reflection factorization we have an associated path in the Bruhat graph. The first goal of this section is to provide a ``normal form'' for paths in the Bruhat graph attached to certain reflection factorizations (see Proposition \ref{prop:key_prop}). Its proof is based on the following two results.

\begin{lemma}\label{lem:hurw_dihedral}
Let $(W,S)$ be a Coxeter system of finite rank, $t_{1},t_{2}\in T$ with $t_{1}\neq t_{2}$. Then $(r,s)\in \left(T\cap \langle t_{1},t_{2}\rangle\right)^{2}$ and $(t_{1},t_{2})$ lie in the same Hurwitz orbit if and only if $rs=t_{1}t_{2}$.
\end{lemma}
\begin{proof}
A direct computation shows for $m\in \ZZ_{\geq 0}$ that
\begin{align*}
    \sigma_{1}^{m}(t_{1},t_{2})&=\left((t_{1}t_{2})^{m}t_{1},~t_{1}(t_{2}t_{1})^{m-1}\right) \text{ and }\\
     \sigma_{1}^{-m}(t_{1},t_{2})&=\left(t_{2}(t_{1}t_{2})^{m-1},~(t_{2}t_{1})^{m}t_{2}\right).
\end{align*}
Since $r\in T\cap \langle t_{1},t_{2}\rangle$ there exists $m\in \ZZ_{\geq 0}$ with $r\in \lbrace (t_{1}t_{2})^{m}t_{1},~(t_{2}t_{1})^{m}t_{2} \rbrace$. The latter implies that $(r,s)\sim (t_{1},t_{2})$ if and only if $rs=t_{1}t_{2}$.
\end{proof}

\begin{lemma} \label{lem:changing_path}
Let $(W,S)$ be a Coxeter system of finite rank, $w\in W$ and $t_{1},t_{2}\in T$ with $t_{1}\neq t_{2}$ such that 
\[w\longrightarrow wt_{1}\longleftarrow wt_{1}t_{2}\]
in $\Omega_{(W,S)}$. Then there exists $(t'_{1},t'_{2})\in \B_{2}(t_{1},t_{2})$ such that one of the following cases hold:
\begin{enumerate}
    \item[(a)] $w\longrightarrow wt'_{1}\longrightarrow wt'_{1}t'_{2}=wt_{1}t_{2}$
    \item[(b)] $w\longleftarrow wt'_{1}\longleftarrow wt'_{1}t'_{2}=wt_{1}t_{2}$
    \item[(c)] $w\longleftarrow wt'_{1}\longrightarrow wt'_{1}t'_{2}=wt_{1}t_{2}$
\end{enumerate}
In particular, in all three cases we have $\ell_{S}(wt'_{1})<\ell_{S}(wt_{1})$.
\end{lemma}
\begin{proof}
Let $W'=\langle t_{1},t_{2}\rangle$ and $S'=\chi(W')$, then $(W',S')$ is a Coxeter system of rank two. The isomorphisms (\ref{eq:a2}) map $w\longrightarrow wt_{1}\longleftarrow wt_{1}t_{2}$ in $\Omega_{(W,S)}$ to $x\longrightarrow xt_{1}\longleftarrow xt_{1}t_{2}$ in $\Omega_{(W',S')}$ for some $x\in W'$. If $x=e$, we choose an arbitrary $t'_{1}\in S'$. Since $t_{1}\neq t_{2}$ we get with $t'_{2}=t'_{1}t_{1}t_{2}\in T\cap \langle t_{1},t_{2}\rangle$
\[x\longrightarrow xt'_{1}\longrightarrow xt'_{1}t'_{2}.\]
If $x\neq e$, there exists $t'_{1}\in S'$ such that we have either
\begin{align*}
x\longleftarrow xt'_{1} \longleftarrow xt'_{1}t'_{2} &,~\text{or}\\ 
x\longleftarrow xt'_{1} \longrightarrow xt'_{1}t'_{2} & {}
\end{align*}
for $t'_{2}=t'_{1}t_{1}t_{2}\in T\cap \langle t_{1},t_{2}\rangle$. Hence the isomorphisms (\ref{eq:a2}) yield one of the paths of $\Omega_{(W,S)}(wW')$ described in (a), (b) or (c). Moreover, Lemma \ref{lem:hurw_dihedral} implies that $(t_{1},t_{2})$ and $(t'_{1},t'_{2})$ lie in the same Hurwitz orbit. 

Next we compare the length of $wt'_{1}$ with the length of $wt_{1}$. In case (a) we have
\[\ell_{S}(wt'_{1})<\ell_{S}(wt'_{1}t'_{2})=\ell_{S}(wt_{1}t_{2})<\ell_{S}(wt_{1})\]
while in cases (b) and (c) we have
\[\ell_{S}(wt'_{1})<\ell_{S}(w)<\ell_{S}(wt_{1}),\]
as desired.
\end{proof}

Next we connect the Hurwitz action with the Bruhat graph. It can be interpreted as some kind of normal form for paths attached to reflection factorizations.

\begin{proposition}\label{prop:key_prop}
Let $(W,S)$ be a Coxeter system of finite rank, $x,w\in W$ and $w=t_{1}\cdots t_{m}$ a reflection factorization such that each factorization of $\B_{m}(t_{1},\ldots,t_{m})$ consists of pairwise different factors. Then there exists a factorization $(t'_{1},\ldots,t'_{m})\in \B_{m}(t_{1},\ldots,t_{m})$ such that the corresponding path in the Bruhat graph starting in $x$ and ending in $xw$ is first decreasing and then increasing. More precisely, we have
\[x\longleftarrow xt'_{1}\longleftarrow \ldots \longleftarrow xt'_{1}\cdots t'_{i}\longrightarrow \ldots \longrightarrow xt'_{1}\cdots t'_{m}=xw\]
for a unique $0\leq i \leq m$.
\end{proposition}
\begin{proof}
Consider the undirected path in $\Omega_{(W,S)}$ corresponding to the reflection factorization $(t_{1},\ldots,t_{m})$ of $w\in W$
\[x\edge xt_{1}\edge xt_{1}t_{2}\edge \ldots \edge xt_{1}\ldots t_{m}=xw.\]
Since every factorization of $\B_{m}(t_{1},\ldots,t_{m})$ contains pairwise different reflections, Lemma \ref{lem:changing_path} allows us to change parts of the associated directed path of shape $\star \longrightarrow \star \longleftarrow \star$ to 
\[\star \longrightarrow \star \longrightarrow \star,~~ \star \longleftarrow \star \longleftarrow \star~~  \text{ or }~~ \star \longleftarrow \star \longrightarrow \star\]
only using the Hurwitz action. Also by Lemma \ref{lem:changing_path} each replacement reduces the sum of the length of the vertices. Eventually, after finitely many steps we obtain a path that is first decreasing and then increasing.
\end{proof}

In case of reduced reflection factorizations the previous statement yields the following.

\begin{corollary}[{\cite[Proposition 2.2]{BDSW14}}] \label{cor:ordering}
Let $(W,S)$ be a Coxeter system of finite rank, $w\in W$ and $(t_{1},\ldots,t_{m})$ a reduced reflection factorization of $w$, then there exists $(t'_{1},\ldots,t'_{m})\in \B_{m}(t_{1},\ldots,t_{m})$ such that the corresponding path in the Bruhat graph is strictly increasing, that is,
\[e\longrightarrow t'_{1}\longrightarrow t'_{1}t'_{2}\longrightarrow \ldots \longrightarrow t'_{1}t'_{2}\cdots t'_{m}=w.\]
\end{corollary}

\begin{lemma}\label{lem:set_of_refl}
Let $(W,S)$ be a Coxeter system of finite rank, $t,t'\in T$ $(t\neq t')$ and $W'=\langle t_{1},\ldots,t_{m}\rangle $ a reflection subgroup ($t_{i}\in T$, $1\leq i \leq m$). The multiset of conjugacy classes of $\lbrace t,t'\rbrace$ and $\chi(\langle t,t' \rangle)$  under conjugation with elements from $\langle t,t' \rangle$ coincide. The set of reflections for $(W',\chi(W'))$ is $W'\cap T=\bigcup_{w\in W'}w\lbrace t_{1},\ldots,t_{m} \rbrace w^{-1}$.
\end{lemma}
\begin{proof}
Let $\chi(\langle t,t' \rangle)=\lbrace r_{1},r_{2}\rbrace$. By Lemma \ref{lem:hurw_dihedral} we have that 
\begin{align*}
    (r_{1},\star)\sim (t,t')\sim (r_{2},\star).
\end{align*}
Hence $\lbrace t,t' \rbrace$ and $\chi(\langle t,t' \rangle)$ have the same multisets of conjugacy classes.

To prove the second part of the lemma we use the algorithm described Lemma \ref{le:DyerAlg}. Set $T_{0}=\lbrace t_{1},\ldots,t_{m} \rbrace$. In any step $i\geq 0$ of the algorithm we exchange $\lbrace t,t' \rbrace \subseteq T_{i}$ $(t\neq t')$ by $\chi(\langle t,t' \rangle)$ until we reach $\chi(W')$. Since the multisets of conjugacy classes of $T_{i}$ and $T_{i+1}$ coincide for any $i\in \ZZ_{\geq 0}$ we have
\[\bigcup_{w\in W'} w T_{i} w^{-1}=\bigcup_{w\in W'}w T_{i+1} w^{-1}.\]
After finitely many steps the algorithm yields the canonical simple system $T_{k}=\chi(W')$ for $k \gg 0$. Hence
\[\bigcup_{w\in W'} w\chi(W') w^{-1}=\bigcup_{w\in W'}w \lbrace t_{1},\ldots,t_{m} \rbrace w^{-1}.\]
Now the fact (\ref{eq:set_of_refl}) yields the assumption
\[\bigcup_{w\in W'}w \lbrace t_{1},\ldots,t_{m} \rbrace w^{-1}=T\cap W'.\]
\end{proof}

We close this section with a proposition that investigates the canonical simple system of reflection subgroups of rank two. It will be useful in the investigation of quasi-Coxeter elements in Section \ref{sec:quasi_weyl}.

\begin{proposition}\label{prop:root_in_parab_subgrp}
Let $(W,S)$ be a Coxeter system of finite rank and $P$ a standard parabolic subgroup of $W$. If $t\in P\cap T$ and $t'\in T \setminus P$, then $t\in \chi(<t,t'>)$.
\end{proposition}

\begin{proof}
Since $P$ is a standard parabolic subgroup Lemma \ref{lem:set_of_refl} implies
\[t\in T\cap P=\bigcup_{w\in P}w(S\cap P)w^{-1}.\]
Let $w\in P$ such that $s':=wtw^{-1}\in S$ and $\ell_{S}(w)$ be minimal among all those $w\in P$. Therefore we get
\[s'\in \chi(<s',wt'w^{-1}>)=\chi(<wtw^{-1},wt'w^{-1}>)=\chi(w<t,t'>w^{-1}).\]
If $w=e$ the previous equation yields the assertion. Thus assume that $\ell_{S}(w)\geq 1$ and set $t'':=wt'w^{-1}$. Let $w=s_{1}\cdots s_{n}$ be a reduced factorization in simple reflections. Because of the minimality of $\ell_{S}(w)$ we have $s_{i}s_{i-1}\cdots s_{1}s's_{1}\cdots s_{i-1}s_{i}\notin S$ for all $1\leq i \leq n$. The latter can be verified as follows. Assume that $s_{i}s_{i-1}\cdots s_{1}s's_{1}\cdots s_{i-1}s_{i}=:s''\in S$ for some $1\leq i\leq n$. Then 
\[s_{n}\cdots s_{i+1}s''s_{i+1}\cdots s_{n}=s_{n}\cdots s_{i+1}s_{i}\cdots s_{1}s's_{1}\cdots s_{i}s_{i+1}\cdots s_{n}=w^{-1}s'w=t\]
and thus 
\[s_{i+1}\cdots s_{n}ts_{n}\cdots s_{i+1}=s''\in S.\]
The latter contradicts the minimality of $\ell_{S}(w)$.

In the following we will show by induction that
\[s_{i}\cdots s_{1}\chi(\langle s',t'' \rangle)s_{1}\cdots s_{i}=\chi(s_{i}\cdots s_{1}\langle s',t''\rangle s_{1}\cdots s_{i})\]
for all $1\leq i \leq n$.  Consider the situation for $i=1$. Since $s'\in S\cap P$ and $t''\notin P$ we have $\chi(\langle s',t''\rangle)=\lbrace s',r\rbrace$ for some reflection $r\notin P$. Since $s_{1}s's_{1}\notin S$ and $r\notin P$ we have $s_{1}\notin \lbrace s',r \rbrace=\chi(\langle s',t'' \rangle)$. Hence by the equation (\ref{eq:a1}) we get $s_{1}\chi(\langle s',t'' \rangle)s_{1}=\chi(s_{1}\langle s',t''\rangle s_{1})$. Assume that $i\geq 2$. By induction hypothesis it holds
\begin{align*}
\chi(s_{i-1}\cdots s_{1}<s',t''>s_{1} \cdots s_{i-1})&=s_{i-1}\cdots s_{1}\chi(<s',t''>)s_{1} \cdots s_{i-1}\\
&=s_{i-1}\cdots s_{1}\lbrace s',r \rbrace s_{1} \cdots s_{i-1}.
\end{align*}
As before we have that $s_{i}\notin  s_{i-1}\cdots s_{1}\lbrace s',r \rbrace s_{1} \cdots s_{i-1} = \chi(s_{i-1}\cdots s_{1} \langle s',t''\rangle s_{1} \cdots s_{i-1})$. Thus equation (\ref{eq:a1}) implies
\[s_{i}\chi(s_{i-1}\cdots s_{1} \langle s',t''\rangle s_{1} \cdots s_{i-1})s_{i}=\chi(s_{i}\cdots s_{1} \langle s',t''\rangle s_{1} \cdots s_{i}).\]
The latter yields by the induction hypothesis that
\[\chi(s_{i}\cdots s_{1} \langle s',t''\rangle s_{1} \cdots s_{i})=s_{i}\cdots s_{1}\chi(\langle s',t'' \rangle)s_{1}\cdots s_{i}.\] Altogether, we have with $s'\in \chi(\langle s',t''\rangle)$
\[s_{i}\cdots s_{1}s's_{1}\cdots s_{i} \in \chi(s_{i}\cdots s_{1}<s',t''>s_{1} \cdots s_{i})\]
for all $1\leq i \leq n$. In particular, for $i=n$ we have
\[t=w^{-1}s'w\in \chi(w^{-1}<s',t''>w)=\chi(<t,t'>).\]
\end{proof}

\section{Reflection factorizations in finite Coxeter groups}

The aim of this section is to investigate arbitrary reflection factorizations in finite Coxeter groups, and to provide a method to reduce these factorizations. We give a uniform proof of Theorem \ref{thm:main1}. Moreover, we state in Proposition \ref{prop:alt_main1} a version of this theorem in a more general setup.

We start with the following well-known lemma.

\begin{lemma}\label{lem:rk_par_cl}
Let $(W,S)$ be a finite Coxeter system and $w=t_{1}\cdots t_{m}\in W$ a reduced reflection factorization of $w$. Then the group $W'=\langle t_{1},\ldots,t_{m}\rangle$ is a Coxeter group of rank $m$ and its rank coincides with the rank of the parabolic closure $P(t_{1},\ldots,t_{m})$ of $\lbrace t_{1},\ldots,t_{m}\rbrace$. 
\end{lemma}
\begin{proof}
By \cite[Corollary 3.11]{Dye90} the group $W'$ is a Coxeter group of rank at most $m$. Since the factorization $t_{1}\cdots t_{m}$ is reduced, Carter's lemma \ref{lem:Carter} implies that the rank is at least $m$. By \cite[Lemma 2.1]{DPR13} the rank of the parabolic closure of $\lbrace t_{1},\ldots,t_{m}\rbrace$ coincides with the rank of $W'$.
\end{proof}

We are now in the position to prove Theorem \ref{thm:main1}. This proof describes a procedure that allows to simplify reflection factorizations in finite Coxeter groups. It was first proven by Lewis and Reiner \cite[Corollary 1.4]{LR16} by a case-based analysis of the finite irreducible Coxeter groups. We provide a uniform proof here.

\begin{proof}[Proof of Theorem \ref{thm:main1}]
We proceed by induction on $k$. If $k=0$ there is nothing to prove. Assume that $k>0$ and let $1\leq l \leq m+2k-1$ maximal such that $w':=t_{1}\cdots t_{l}$ is a reduced reflection factorization. Since $\ell_{T}(w' t_{l+1})=l-1$ there exists a factorization $w'=x^{-1}t_{l+1}$ with $x^{-1}\in W$ and $\ell_{T}(w')=\ell_{T}(x^{-1})+1$. By Lemma \ref{lem:rk_par_cl} the rank of the reflection subgroup $W'=\langle t'_{1},\ldots,t'_{l-1}\rangle$ is $l-1$ and coincides with the rank of $P(x^{-1})=P(x)$. Without loss of generality, we can assume that $P(x)$ is a standard parabolic subgroup of $W$.

Consider the (non-directed) path in the Bruhat graph starting in $x$ and ending in $e$ corresponding to the non-reduced factorization $(t_{1},\ldots,t_{l+1})$, that is, the path
\[x \edge xt_{1} \edge xt_{1}t_{2} \edge \ldots  \edge xt_{1}t_{2}\cdots t_{l}=t_{l+1} \edge e.\]
If there exists a factorization in $\B_{l+1}(t_{1},\ldots,t_{l+1})$ with two identical factors, then we can shift them to the end of the factorization by using the Hurwitz action and apply the induction hypothesis. Hence let us assume to the contrary that each factorization in $\B_{l+1}(t_{1},\ldots,t_{l+1})$ consists of pairwise different factors. Then Proposition \ref{prop:key_prop} yields the existence of a braid $\sigma\in \B_{l+1}$ such that the factorization $\sigma(t_{1},\ldots,t_{l+1})=(\overline{t}_{1},\ldots,\overline{t}_{l+1})$ induces the following directed path in the Bruhat graph
\[x\longleftarrow x\overline{t}_{1} \longleftarrow x\overline{t}_{1}\overline{t}_{2}\longleftarrow \ldots \ \longleftarrow x\overline{t}_{1}\cdots \overline{t}_{l}=\overline{t}_{l+1}\longleftarrow e.\]
The strong exchange condition yields that $\overline{t}_{1},\ldots,\overline{t}_{l+1}\in P(x)$ and therefore also $t_{1},\ldots,t_{l+1}\in P(x)$. In particular, $w'\in P(x)$ and
\[l=\ell_{T}(w')\leq \ell_{T\cap P(x)}(w').\]
But by Carter's Lemma \ref{lem:Carter} the length $\ell_{T\cap P(x)}$ is bounded by the rank of $P(x)$, that is, by $l-1$. Hence we arrive at a contradiction and there exists a braid $\sigma\in \B_{m+2k}$ such that 
\[\sigma (t_{1},\ldots,t_{m+2k})=(t'_{1},\ldots, t'_{m+2(k-1)},r_{i_{k}},r_{i_{k}}).\]
The induction hypothesis yields the assertion.
\end{proof}

\begin{remark}
The braid $\sigma$ in Theorem \ref{thm:main1} can be calculated explicitely. As in the proof of Theorem \ref{thm:main1} described, we have to transform the path
\[x \edge xt_{1} \edge xt_{1}t_{2} \edge \ldots  \edge xt_{1}t_{2}\cdots t_{l}=t_{l+1} \edge e\]
into a directed path. This can be done succesively by using Lemma \ref{lem:changing_path} and its proof.
\end{remark}

\begin{remark}
The uniform  proof of Theorem \ref{thm:main1} also yields a uniform proof of a result by Lewis--McCammond--Petersen--Schwer \cite[Theorem B]{LMPS19} about translation-elliptic factorizations in affine Coxeter groups.
\end{remark}

The following calculation shows that Theorem \ref{thm:main1} does not hold for arbitrary Coxeter groups.

\begin{example}
We use the notation of Humphreys' textbook on Coxeter groups \cite[Chapter 4]{Hum90}. Consider the affine Coxeter group of type $\widetilde{B}_{2}$. Further consider the roots $\alpha_{1}=e_{1}-e_{2},\widetilde{\alpha}=e_{1}+e_{2},\alpha_{2}=e_{1}$ of the finite root system $B_{2}\subseteq \RR^{2}$, where $e_{1},e_{2}$ are the canonical unit vectors. We have
\[s_{\alpha_1, 1} s_{\alpha_1} s_{\widetilde{\alpha}, 1} s_{\widetilde{\alpha}}=s_{\alpha_{2},1}s_{\alpha_{2}}.\]
Since $\alpha_1$ is orthogonal to $\widetilde{\alpha}$, every factorization of 
$\B_{4}(s_{\alpha_1, 1},s_{\alpha_1},s_{\widetilde{\alpha}, 1},s_{\widetilde{\alpha}})$ consists of pairwise different factors. 
\end{example}

The following result is a modification of Theorem \ref{thm:main1} that holds for arbitrary Coxeter systems of finite rank. Its short proof is based on Proposition \ref{prop:key_prop} and can be found in \cite{WY19}.

\begin{proposition}[{\cite[Lemma 2.3]{WY19}}] \label{prop:alt_main1}
Let $(W,S)$ be a Coxeter system of finite rank, $w=t_{1}\cdots t_{m+2k}\in W$ a reflection factorization, $k\in \ZZ_{\geq 0}$ and $\ell_{S}(w)=m$. Then there exists a braid $\sigma \in \B_{m+2k}$ such that
\[\sigma (t_{1},\ldots,t_{m+2k})=(r_{1},\ldots,r_{m},r_{i_{1}},r_{i_{1}},\ldots,,r_{i_{k}},r_{i_{k}}).\]
\end{proposition}

\section{Extension of reduced reflection factorizations in arbitrary Coxeter groups}

In this section we show how to extend reduced reflection factorizations in arbitrary Coxeter systems of finite rank such that they all lie in the same Hurwitz orbit. As a consequence we get that all reduced reflection factorizations are subwords of a factorization that lies in the Hurwitz orbit of a reduced factorization in simple reflections.

\begin{definition}
Let $(W,S)$ be a Coxeter system with set of reflections $T$. For $w \in W$ we define the set $N(w):= \{t \in T \mid \ell(wt)< \ell(w) \}$.
\end{definition}

\begin{lemma}\label{lem:fac_in_N}
Let $(W,S)$ be a Coxeter system of finite rank, $w\in W$ and $(t_{1},\ldots,t_{n})$ be a reduced reflection factorization of $w$. Then there exists $(r_{1},\ldots,r_{n})\in \B_{n}(t_{1},\ldots,t_{n})$ such that $r_{i}\in N(w)$ for all $1\leq i \leq n$.
\end{lemma}
\begin{proof}
Let $w=s_{1}\cdots s_{m}$ be a reduced factorization in simple reflections. By \cite[Proposition 5.6]{Hum90} we have
\[N(w)=\lbrace s_{m},s_{m}s_{m-1}s_{m},\ldots,s_{m}\cdots s_{1} \cdots s_{m} \rbrace,\]
which is independent of the initial reduced factorization. Corollary \ref{cor:ordering} implies the existence of a factorization $(r_{1},\ldots,r_{n})\in \B_{n}(t_{1},\ldots,t_{n})$ such that $\ell(wr_{n}\cdots r_{i})<\ell(wr_{n}\cdots r_{i+1})$ for $1 \leq i \leq n$. Thus the strong exchange condition yields $1\leq i_{n-1} \neq i_{n}\leq m$ such that $wr_{n}=s_{1}\cdots \widehat{s_{i_{n}}} \cdots s_{m}$ and 
$$
wr_{n}r_{n-1}=s_{1}\cdots \widehat{s_{i_{n}}} \cdots \widehat{s_{i_{n-1}}} \cdots s_{m} \quad \text{or} \quad wr_{n}r_{n-1}=s_{1}\cdots \widehat{s_{i_{n-1}}} \cdots \widehat{s_{i_n}} \cdots s_{m}.
$$
If we have $i_{n}<i_{n-1}$, that is, $wr_{n}r_{n-1}=s_{1}\cdots \widehat{s_{i_{n}}} \cdots \widehat{s_{i_{n-1}}} \cdots s_{m}$, then $r_{n-1},r_{n}\in N(w)$. Otherwise apply the Hurwitz move induced by $\sigma_{n-1}^{-1}\in \B_{n}$ to $(r_{1},\ldots,r_{n})$ and obtain
\[\sigma_{n-1}^{-1}(r_{1},\ldots,r_{n})=(r_{1},\ldots,r_{n-2},r_{n},r_{n}r_{n-1}r_{n}).\]
Since $wr_{n}=s_{1}\cdots \widehat{s_{i_{n}}}\cdots s_{m}$, $wr_{n}r_{n-1}r_{n}=s_{1}\cdots \widehat{s_{i_{n-1}}}\cdots s_{m}$ and $i_{n-1}<i_{n}$ we get $r_{n},r_{n}r_{n-1}r_{n}\in N(w)$.

Now we proceed in this way for all the neighbors of the resulting factorization until we obtain (after finitely many steps) a factorization whose reflections are in $N(w)$.
\end{proof}

The proof of the previous lemma shows that for a given reduced factorization $w=s_{1}\cdots s_{m}$ of $w$ in simple reflections, we find $(r_{1},\ldots,r_{n})\in \B_{n}(t_{1},\ldots,t_{n})\subseteq \Red_{T}(w)$ such that 
\[wr_{n}\cdots r_{k}=s_{1}\cdots \widehat{s_{i_{n}}}\cdots \widehat{s_{i_{k}}} \cdots s_{n}\]
for $1\leq k \leq n$ and $1\leq i_{n}<i_{n-1}<\ldots<i_{1}\leq n$. In particular, we get $r_{k}=s_{n}\cdots s_{i_{k}}\cdots s_{n}$ for $1\leq k \leq n$.

\begin{proof}[Proof of Theorem \ref{thm:main1.1}]
By the previous mentioned we can assume that 
\[t_{k}=s_{m}\cdots s_{i_{k}}\cdots s_{m}\] for $1\leq k \leq n$ and $1\leq i_{n}<i_{n-1}<\ldots<i_{1}\leq n$.

Let $(q_{1},\ldots,q_{m-n})=(s_{1},\ldots,\widehat{s_{i_{n}}},\ldots,\widehat{s_{i_{1}}},\ldots,s_{m})$, that is, the $(m-n)$-tuple that is obtained by deleting the entries of $(s_{1},\ldots,s_{m})$ with indices $i_{1},\ldots,i_{n}$. Since $t_{k}=s_{m}\cdots s_{i_{k}}\cdots s_{m}$ we obtain
\[(q_{1},\ldots,q_{m-n},t_{1},\ldots,t_{n})=(s_{1},\ldots,\widehat{s_{i_{n}}},\ldots,\widehat{s_{i_{1}}},\ldots,s_{m},s_{m}\cdots s_{i_{1}}\cdots s_{m},\ldots, s_{m}\cdots s_{i_{n}}\cdots s_{m}).\]
In the following we show by induction that $(q_1,\ldots,q_{m-n},t_1,\ldots,tn)$ and $(s_1,\ldots,s_m)$ lie in the same Huritz orbit. If $n=1$ we have
\[(q_{1},\ldots,q_{m-1},t_{1})=(s_{1},\ldots,\widehat{s_{i_{1}}},\ldots,s_{m},s_{m}\cdots s_{i_{1}}\cdots s_{m}).\]
A direct calculation shows
\[\sigma_{i_{1}}\cdots \sigma_{m-1}(s_{1},\ldots,\widehat{s_{i_{1}}},\ldots,s_{m},s_{m}\cdots s_{i_{1}}\cdots s_{m})=(s_{1},\ldots,s_{m}).\]
Now assume that $n>1$. Similar to the case $n=1$ we have
\begin{align*}
&\sigma_{i_{1}}\cdots \sigma_{m-1}(q_{1},\ldots,q_{m-n},t_{1},\ldots,t_{n})\\
=&\sigma_{i_{1}}\cdots \sigma_{m-1}(s_{1},\ldots,\widehat{s_{i_{n}}},\ldots,\widehat{s_{i_{1}}},\ldots,s_{m},s_{m}\cdots s_{i_{1}}\cdots s_{m},\ldots, s_{m}\cdots s_{i_{n}}\cdots s_{m})\\
=&(s_{1},\ldots,\widehat{s_{i_{n}}},\ldots,\widehat{s_{i_{2}}},\ldots,s_{m},s_{m}\cdots s_{i_{2}}\cdots s_{m},\ldots, s_{m}\cdots s_{i_{n}}\cdots s_{m})
\end{align*}
The induction hypothesis yields the assumption.
\end{proof}

\begin{corollary}
Let $(W,S)$ be a Coxeter system of finite rank and $w=s_{1}\ldots s_{m}$ reduced factorization of $w\in W$. Every reduced reflection factorization of $w$ is a prefix of an element of $\B_{m}(s_{1},\ldots,s_{m})$. Moreover, given two reduced reflection factorizations $(r_{1},\ldots,r_{n}),(t_{1},\ldots,t_{n})$ of $w$ then there exist $q_{1},\ldots,q_{m-n},p_{1},\ldots,p_{m-n}\in S$ with 
\[q_{1}\cdots q_{m-n}=e=p_{1}\cdots p_{m-n}\]
such that $(q_{1},\ldots,q_{m-n},t_{1},\ldots,t_{n})$ and $(p_{1},\ldots,p_{m-n},r_{1},\ldots,r_{n})$ lie in the same Hurwitz orbit.
\end{corollary}

\section{Quasi-Coxeter elements in Weyl groups}\label{sec:quasi_weyl}

In this section we investigate uniformly the so-called quasi-Coxeter elements, which are a generalization of Coxeter elements. 
We deduce a case-free proof of \cite[Theorem 6.1]{LR16} for Weyl groups that determines the Hurwitz orbits of arbitrary reflection factorizations of quasi-Coxeter elements and that is already proven for Coxeter elements in arbitrary Coxeter groups of finite rank in \cite{WY19}. Most results of this section hold for quasi-Coxeter elements in arbitrary Coxeter systems of finite rank. We start with the definition of a (parabolic) quasi-Coxeter element.

\begin{definition}
Let $(W,S)$ be a Coxeter system of rank $n\in \NN$. An element $c\in W$ is called \defn{Coxeter element} if it is conjugated to an element that admits a factorization $c=s_{i_{1}}\cdots s_{i_{k}}$ in pairwise different simple reflections with $k=n$, and it is called \defn{parabolic Coxeter element} if $k\leq n$. An element $w\in W$ is called \defn{quasi-Coxeter element} if it admits a reduced reflection factorization with $n$ factors which generate $W$ and it is called \defn{parabolic quasi-Coxeter element} if it is a quasi-Coxeter element for a parabolic subgroup. It is called a \defn{proper} parabolic quasi-Coxeter element if it is a quasi-Coxeter element for a proper parabolic subgroup.
\end{definition}

\begin{example}
Every conjugate of a (parabolic) quasi-Coxeter element is a (parabolic) quasi-Coxeter element. Parabolic Coxeter elements are by definition parabolic quasi-Coxeter elements, but there exist quasi-Coxeter elements which are not conjugated to Coxeter elements, for instance see \cite[Example 2.4]{BGRW17}.
\end{example}

\begin{remark}
Let $W_{\Phi}$ be an irreducible simply laced Weyl group, that is, $\Phi$ is of type $A_{n}$, $D_{n}$ ($n \in \NN$) or $E_{6}, E_7, E_8$. Quasi-Coxeter elements were first defined by Voigt \cite{Voi85} in a slightly different way. He defines an element $w=s_{\alpha_{1}}\cdots s_{\alpha_{n}}\in W$ to be quasi-Coxeter if the $\ZZ-$span of the roots $\alpha_{1},\ldots,\alpha_{n}$ equals the root lattice of $\Phi$. By \cite[Lemma 5.12]{BGRW17} the two notions of quasi-Coxeter elements coincide.
\end{remark}

The next results are the first approach towards a uniform proof of the transitive Hurwitz action on the set of reduced reflection factorizations of quasi-Coxeter elements in Weyl groups. Nevertheless, if possible, we state and prove these results in a more general setting.

\begin{lemma}
Let $(W,S)$ be a Coxeter system of rank $n\in \NN$ and $t_{1},\ldots,t_{n}\in T$ such that $\langle t_{1},\ldots,t_{n} \rangle=W$. Then the multisets of conjugacy classes of $S$ and $\lbrace t_{1},\ldots,t_{n} \rbrace$ coincide.
\end{lemma}
\begin{proof}
If $n=1$ the assumption is obviously satisfied. Thus assume that $n\geq 2$. In the following we will use the algorithm described in Lemma \ref{le:DyerAlg}. Set $T_{0}=\lbrace t_{1},\ldots,t_{n}\rbrace$. In the first step of the algorithm we exchange two different reflections $t_{i},t_{j}\in T_{0}$ for $i<j$ by $\chi(\langle t_{i},t_{j} \rangle)$ to get $T_{1}$. By Lemma \ref{lem:set_of_refl} the multisets of the conjugacy classes of $\lbrace t_{i},t_{j} \rbrace$ and $\chi(\langle t_{1},t_{2} \rangle)$ coincide. Thus the multisets of conjugacy classes of $T_{0}$ and $T_{1}$ coincide. Inductively we get that the multisets of conjugacy classes of all the $T_{i}$ for $i\geq 0$ are the same. Since $\langle T_{0}\rangle=W$ the algorithm terminates in $S$ after finitely many steps, i.e. there exists a $m\in \ZZ_{\geq 0}$ such that $T_{m}=S$. The latter implies the assumption. 
\end{proof}

The following theorem can be deduced from Proposition \ref{prop:key_prop} by using the strong exchange condition for Coxeter groups.

\begin{theorem}[{\cite[Theorem 1.4]{BDSW14}}]\label{thm:fac_in_par_subg}
Let $(W,S)$ be a Coxeter system of finite rank, $P$ a parabolic subgroup and $w=t_{1}\ldots t_{m}\in P$ a reduced reflection factorization. Then $t_{1},\ldots,t_{m}\in P$.
\end{theorem}

As a consequence of the previous theorem we get the following result.

\begin{lemma}[{\cite[Proposition 2.5]{HW20}}]\label{lem:parab_cl}
Let $(W,S)$ be a Coxeter system of finite rank, $(t_{1},\ldots, t_{m})
$ a reduced reflection factorization of $w\in W$. Then the parabolic closure $P(w)$ of $w$ coincides with the parabolic closure $P(t_{1},\ldots,t_{m})$ of $t_{1},\ldots,t_{m}$. 
\end{lemma}

Another common property of Coxeter elements and quasi-Coxeter element is the following connection between the order of a quasi-Coxeter element and the order of the ambient Coxeter group. Since Coxeter elements are quasi-Coxeter elements the following result is a new proof of the classical fact proven in \cite{Ho82}.

\begin{corollary} \label{cor:FinOrdQuasi}
Let $(W,S)$ be an irreducible Coxeter system of finite rank and $w\in W$ a quasi-Coxeter element. The order of $w$ is finite if and only if $W$ is finite.
\end{corollary}
\begin{proof}
By Lemma \ref{lem:parab_cl} it holds $P(w)=W$. Elements with the previous property are called essential. 
Assume that $W$ is infinite. By \cite[Corollary 2.5]{Pa07} we have that $w^{p}$ is essential for all $p\in \NN$. In particular, $w^{p}\neq e$ for all $p\in \NN$.
\end{proof}

\begin{definition} \label{def:AbsOrder}
Let $(W,S)$ be a Coxeter system with set of reflections $T$. We define a partial order $\leq_T$ on $W$, called \defn{absolute order}, by setting for $u,v \in W$:
$$
u \leq_T v \iff \ell_T(u) + \ell_T(u^{-1}v) = \ell_T(v).
$$
\end{definition}

Equivalently, we have $u \leq_T v$ if and only if there exists $(t_1, \ldots , t_m) \in \Red_T(v)$ and some $k \leq m$ such that $(t_1, \ldots , t_k) \in \Red_T(u)$.

\medskip

Theorem \ref{prop:prefix_quasi_cox} provides a characterization of \textit{maximal} parabolic quasi-Coxeter elements in Coxeter systems of finite rank $n$, that is, parabolic quasi-Coxeter elements of reflection length $n-1$. We state its proof.

\begin{proof}[Proof of Theorem \ref{prop:prefix_quasi_cox}]
For the \textit{only if} direction note that by assumption there exists a reduced reflection factorization $(r_1, \ldots, r_{n-1}) \in \Red_T(x)$ sucht that $P:=\langle r_1, \ldots, r_{n-1} \rangle$ is parabolic. In particular, $P(x) \subseteq P \not\subseteq W$, where the latter inclusion is proper since $W$ cannot be generated by less than $n$ reflections (see \cite[Proposition 2.1]{BGRW17}). Furthermore, since $P$ is parabolic, there exists $r_n \in T$ with $\langle P, r_n \rangle = W$. If we set $w=r_1 \cdots r_{n-1}r_n$, we have that $w$ is a quasi-Coxeter element and $x \leq_T w$ as desired.

We show the \textit{if} direction. Let $(t_{1},\ldots,t_{n})$ be a reduced reflection factorization of $w$ such that $\langle t_{1},\ldots,t_{n} \rangle=W$, $w=tx$ with $\ell_{T}(w)=\ell_{T}(x)+1$. Without loss of generality, we can assume that $P(x)$ is a standard parabolic subgroup of $W$. Assume that each factorization in the orbit
\[\B_{n+1}(t,t_{1},\ldots,t_{n})\subseteq \text{Fac}_{T,n+1}(x):=\lbrace (r_{1},\ldots,r_{n+1})\in T^{n+1}\mid r_{1}\cdots r_{n+1}=x \rbrace\]
contains pairwise different factors. By Corollary \ref{cor:ordering} there exists a reflection factorization $(t',t'_{1},\ldots,t'_{n})\in \B_{n+1}(t,t_{1},\ldots,t_{n})$ that corresponds to the following directed path in the Bruhat graph 
\[x\longleftarrow xt'_{n}\longleftarrow xt'_{n}t'_{n-1}\longleftarrow\ldots \longleftarrow xt'_{n}t'_{n-1}\cdots t'_{1}=t'\longleftarrow e.\]
Since $P(x)$ is a proper standard parabolic subgroup, the strong exchange condition yields 
$W=\langle t,t_{1},\ldots,t_{n} \rangle=\langle t',t'_{1},\ldots,t'_{n} \rangle\subseteq P(x)$,
a contradiction. Thus there exists a reflection factorization $(t'_{1},\ldots,t'_{n-1},t',t')\in \B_{n+1}(t,t_{1},\ldots,t_{n})$ with $(t'_{1},\ldots,t'_{n-1})\in \text{Red}_{T}(x)$.
In particular, we have $\langle t'_{1},\ldots,t'_{n-1},t' \rangle=W$ and thus 
\[\chi(\langle t'_{1},\ldots,t'_{n-1},t' \rangle)=\chi(W)=S.\]
We use the algorithm described in Lemma \ref{le:DyerAlg}. We start with the set $T_{0}:=\chi(t'_{1},\ldots,t'_{n-1})\cup \lbrace t' \rbrace$. Since $\chi(\langle t'_{1},\ldots,t'_{n-1},t' \rangle)=S$ it terminates in $T_{k}=\chi(W)=S$ for $k\gg 0$. In each step $i\in \ZZ_{\geq 0}$ it exchanges $\lbrace r_{1},r_{2}\rbrace \subseteq T_{i}$ $(r_{1}\neq r_{2})$ by $\chi(\langle r_{1},r_{2} \rangle)$. Since $t'\notin P(x)$ and $\langle \chi(t'_{1},\ldots,t'_{n-1})\rangle = \langle t'_{1},\ldots,t'_{n-1} \rangle \subseteq P(x)$, Lemma \ref{prop:root_in_parab_subgrp} yields that $\chi(\langle t'_{1},\ldots, t'_{n-1} \rangle)\subseteq T_{i}$ for all $i\geq 0$. In particular, \[S=T_{k}=\chi(\langle t'_{1},\ldots, t'_{n-1} \rangle) \cup \lbrace s \rbrace\]
with $s\in T$ and $k\gg 0$. Thus we get that 
\[\chi(\langle t'_{1},\ldots, t'_{n-1} \rangle)\subseteq \langle t'_{1},\ldots, t'_{n-1} \rangle \cap P(x)\cap S\]
and $|\chi(\langle t'_{1},\ldots, t'_{n-1} \rangle)|=n-1$. Since $P(x)$ is a proper standard parabolic subgroup we have $P(x)\cap S=\chi(\langle t'_{1},\ldots, t'_{n-1} \rangle$ and therefore $P(x)=\langle t'_{1},\ldots, t'_{n-1} \rangle$.
\end{proof}

\begin{question}
Is Theorem \ref{prop:prefix_quasi_cox} still true if only considering (parabolic) Coxeter elements? More precisely, is the following statement true:

Let $(W,S)$ be a Coxeter system of rank $n\in \NN$ and $x \in W$ with $\ell_T(x)= n-1$. Then the element $x$ is a parabolic Coxeter element if and only if there exists a Coxeter element $c \in W$ with $x \leq_T c$ and $P(x)\neq W$.
\end{question}

\begin{remark}
In the general case, the assumption $P(x) \neq W$ in Theorem \ref{prop:prefix_quasi_cox} is necessary as shown for instance in \cite[Example 5.7]{HK16} while for finite Coxeter groups it is redundant (see Corollary \ref{cor:FinCharQuasi} below).
\end{remark}

\begin{lemma}[{\cite[Theorem 12.3.4 (i)]{Davis}}] \label{lem:Tits}
Let $(W,S)$ be a Coxeter system of finite rank. Then any finite subgroup of $W$ is contained in a finite parabolic subgroup.
\end{lemma}

\begin{remark}
Let $(W,S)$ be an affine Coxeter system, $c \in W$ a Coxeter element and $x\leq_{T}c$. By Lemma \ref{lem:Tits}, the element $x$ is elliptic if and only if $P(x)$ is a proper parabolic subgroup. In this case, by \cite[Theorem 3.22 (i)]{PS19}, the element $x$ is a parabolic Coxeter element and is therefore in particular a parabolic quasi-Coxeter element. The latter is also covered by Corollary \ref{main4}.
\end{remark}



The next lemma shows that a reduced reflection factorization that generates an affine Coxeter group can be transformed by using the Hurwitz action such that the prerequisite of Theorem \ref{prop:prefix_quasi_cox} is satisfied.

\begin{lemma}\label{lem:nice_factorization}
Let $(W,S)$ be a finite or affine Coxeter system of rank $n\in \NN$ and $w=t_{1}\cdots t_{n}\in W$ a reflection factorization of a quasi-Coxeter element such that $\langle t_{1},\ldots,t_{n} \rangle = W$. Then there exists $(r_{1},\ldots,r_{n})\in \B_{n}(t_{1},\ldots,t_{n})$ such that $P(r_{1}\cdots r_{n-1})\neq W$.
\end{lemma}
\begin{proof}
If $(W,S)$ is finite, Lemmata \ref{lem:rk_par_cl} and \ref{lem:parab_cl} imply that $P(t_{1}\cdots t_{n-1})\neq W$. Thus assume that $(W,S)$ is affine and let $p:W\longrightarrow W_{\text{fin}}$ the canonical projection to the corresponding finite Coxeter group $W_{\text{fin}}$ (see (\ref{equ:ProjAffFin})). Since $(W_{\text{fin}},p(S))$ is a finite Coxeter system of rank $n-1$ with set of reflections $p(T)$, the factorization $(p(t_{1}),\ldots,p(t_{n}))$ is not a reduced reflection factorization. By Theorem \ref{thm:main1} there exists a factorization $(r_{1},\ldots,r_{n})\in \B_{n}(t_{1},\ldots,t_{n})$ such that $p(r_{n-1})=p(r_{n})$. Since 
\[W_{\text{fin}}=p(\langle t_{1},\ldots,t_{n} \rangle)=\langle p(t_{1}),\ldots,p(t_{n}) \rangle\]
we get that $W_{\text{fin}}=\langle p(r_{1}),\ldots,p(r_{n-1}) \rangle$, and therefore $(p(r_{1}),\ldots,p(r_{n-1}))$ is a reduced reflection factorization. Thus Carter's Lemma \ref{lem:Carter} implies that the corresponding roots are linear independent and by \cite[Lemma 1.26]{LMPS19} the element $r_{1}\cdots r_{n-1}$ is elliptic. Thus it has finite order, and therefore $P(r_{1}\cdots r_{n-1})$ is finite by Lemma \ref{lem:Tits}. In particular, we get that $P(r_{1}\cdots r_{n-1})\neq W$.
\end{proof}

\begin{remark}
Note that under the assumptions of Lemma \ref{lem:nice_factorization}, the reflection factorization $t_1 \cdots t_n$ is indeed reduced by \cite[Proposition 5.1]{Weg17a}
\end{remark}

We state two direct consequences of Theorem \ref{prop:prefix_quasi_cox} for finite Coxeter groups. Both of which already appear in \cite{BGRW17}. Again we provide uniform proofs.

\begin{corollary}\label{cor:prefix_quasi}
Let $(W,S)$ be a finite Coxeter system of rank $n$ and $t_{1},\ldots,t_{n}\in T$ with $<t_{1},\ldots,t_{n}>=W$, then $<t_{1},\ldots,t_{n-1}>$ is a parabolic subgroup of rank $n-1$.
\end{corollary}

\begin{proof}
By Lemma \ref{lem:rk_par_cl} we have $P(t_1, \ldots, t_{n-1}) \neq W$, hence $t_1\cdots t_{n-1}$ is a parabolic quasi-Coxeter element.
\end{proof}

Theorem \ref{prop:prefix_quasi_cox} is a generalization of \cite[Corollary 6.11]{BGRW17}, which characterizes parabolic quasi-Coxeter elements in finite Coxeter groups. We show that this characterization is a direct consequence of Theorem \ref{prop:prefix_quasi_cox}.

\begin{corollary} \label{cor:FinCharQuasi}
Let $(W,S)$ be a finite Coxeter system with set of reflections $T$ and let $x \in W$. Then $x$ is a parabolic quasi-Coxeter element if and only if there exists a quasi-Coxeter element $w\in W$ such that $x \leq_T w$.
\end{corollary}

\begin{proof}
The \textit{only if} direction is clear by the definition of a parabolic quasi-Coxeter element. The \textit{if} direction follows inductively by Corollary \ref{cor:prefix_quasi}.
\end{proof}

\medskip

The next result is a generalization of \cite[Corollary 6.10]{BGRW17}.

\begin{corollary}
Let $(W,S)$ be a Coxeter system of rank $n$ and $W'$ a reflection subgroup of $W$ of rank $n-1$. Then $W'$ is a parabolic subgroup if and only if it exists $t\in T$ such that $\langle W',t \rangle=W$ and $P(W')\neq W$. 
In particular, if $W$ is finite, then $W'$ is a parabolic subgroup if and only if it exists $t\in T$ such that $\langle W',t \rangle=W$.
\end{corollary}
\begin{proof}
The first part of the statement follows obviously from Theorem \ref{prop:prefix_quasi_cox}. The second part is a consequence of Corollary \ref{cor:prefix_quasi} and Lemma \ref{lem:rk_par_cl}.
\end{proof}

The proof of the following is essentially the proof of \cite[Theorem 1.5]{BW18}.

\begin{proposition}\label{prop:simple_gen}
Let $W$ be a Weyl group with simple system $S$ of rank $n$, $P$ a parabolic subgroup and $t\in T$ such that $\langle P,t \rangle=W$. Then there exist $r_{1},\ldots,r_{n-1}\in P\cap T$ such that $(W,\lbrace r_{1},\ldots,r_{n-1},t \rbrace)$ is a Coxeter system with set of reflection $T$.
\end{proposition}
\begin{proof}
First we will show that beside the trivial case $n=1$ we have that the rank of $P$ is $n-1$. Thus let $n>1$ and $P$ be a proper parabolic subgroup. Assume that the rank of $P$ is at most $n-2$, then the algorithm described in Lemma \ref{le:DyerAlg} with starting set $T_{0}=\chi(P)\cup \lbrace t \rbrace$ and terminal set $T_{k}=S$ for $k\gg 0$ implies the contradictive statement $|S|=|T_{k}|\leq |\chi(P)\cup \lbrace t \rbrace|<n$. Hence we can assume that the rank of $P$ is $n-1$. Moreover, after a suitable conjugation we assume that $P$ is a standard parabolic subgroup of $W$. Let $W\hookrightarrow \text{GL}(V)$ be the geometric representation of the Coxeter system $(W,S)$, $\Phi$ the corresponding root system with associated symmetric bilinear form $(-,-)$, $\Delta_{P}=\lbrace \alpha_{1},\ldots,\alpha_{n-1} \rbrace\subseteq \Phi^{+}$ the canonical simple system for $P$ in sense of \cite[Section 4.1]{DL11} and $s_{\beta}=t$ for $\beta \in \Phi^{+}$. Consider the cone
\[E=\lbrace x\in V \mid (x,\alpha)>0 \text{ for all } \alpha\in \Delta_{P}\cup \lbrace \beta \rbrace\rbrace.\]
Assume that $(\alpha_{j},\beta)<0$ for some $1\leq j \leq n-1$. Then 
\[(\alpha_{j},s_{\alpha_{j}}(\beta))=(s_{\alpha_{j}}(\alpha_{j}),\beta)=-(\alpha_{j},\beta)<0.\]
Then the next calculation shows that the cone 
\[E_{1}=\lbrace x\in V\mid (x,\alpha)>0 \text{ for all } \alpha\in \Delta_{P}\cup \lbrace s_{\alpha_{j}}(\beta) \rbrace\rbrace\]
is contained in $E$. For $x\in E_{1}$ we have
\[(x,\beta)=(x,s_{\alpha_{j}}(s_{\alpha_{j}}(\beta)))=(x,s_{\alpha_{j}}(\beta))-\frac{2(s_{\alpha_{j}}(\beta),\alpha_{j})}{(\alpha_{j},\alpha_{j})}>0.\]
Moreover, this containment is proper. By Lemma \ref{lem:set_of_refl} we get directly that 
\[T=\bigcup_{v\in W}v \lbrace s_{\alpha_{1}},\ldots,s_{\alpha_{n-1}},s_{\beta} \rbrace v^{-1},\]
thus $\spanz(\alpha_{1},\ldots,\alpha_{n-1},\beta)=\spanz(\Phi)$, because $\Phi$ is crystallographic. Hence $\lbrace \alpha_{1},\ldots,\alpha_{n-1},\beta \rbrace$ is linear independent, and thus $M:=\lbrace x\in V \mid (x,\alpha_{i})>0\mid 1\leq i \leq n-1  \rbrace\cap H_{\beta}\neq \emptyset$, where $H_{\beta}$ is the hyperplane perpendicular to $\beta$. For $y\in M$ we have that
\[(y,s_{\alpha_{j}}(\beta))=(s_{\alpha_{j}}(y),\beta)=(y,\beta)-\frac{2(y,\alpha_{j})(\alpha_{j},\beta)}{(\alpha_{j},\alpha_{j})}<0.\]
Thus $y$ is in the closure of $E$, but not in the closure of $E_{1}$, which implies $E_{1}\subset E$.
Recursively we get a strictly descending sequence of cones
\[E\supset E_{1}\supset E_{2} \supset \ldots\]
Since $P$ is finite, the number of cones constructed above is also finite, and therefore this process will stop after finitely many steps. Let $\gamma=w(\beta)$ with $w\in P$ be the root that is obtained in the previous way. Then the pairwise dihedral angles between the roots of $\Delta_{P}\cup\lbrace \gamma \rbrace$ are obtuse. Hence \cite[Lemma 3 (a)]{DL11} yields that $\Delta_{P}\cup\gamma$ is a simple system for $\Phi$, and thus the pair $(W,\lbrace r_{1},\ldots,r_{n-1},t \rbrace)$ is a Coxeter system with $r_{i}=s_{w^{-1}(\alpha_{i})}$ for $1\leq i \leq n-1$.
Moreover, Lemma \ref{lem:set_of_refl} implies directly that $T=\bigcup_{w\in W}w \lbrace r_{1},\ldots,r_{n-1},t \rbrace w^{-1}$.
\end{proof}

Although Proposition \ref{prop:simple_gen} also holds for all dihedral groups, it fails to be true for all finite Coxeter groups.

\begin{example}
Let $(W,\lbrace s_{1},s_{2},s_{3} \rbrace)$ be a Coxeter system of type $H_{3}$ such that the order of $s_{1}s_{2}$ is five and the order of $s_{2}s_{3}$ is three. Set $P=\langle s_{1},s_{3} \rangle$ and $t=s_{2}s_{1}s_{2}$. An easy calculation shows that $\langle P,t \rangle=W$. But $P\cap T=\lbrace s_{1},s_{3} \rbrace$ and the set $\lbrace s_{1},s_{3},t \rbrace$ is not a simple system for $W$.
\end{example}

\begin{corollary}
Let $W$ be a Weyl group with simple system $S$ of rank $n$ and $P$ a parabolic subgroup of rank $n-1$. All the reflections $t\in T$ such that $\langle P,t \rangle =W$ form a single orbit under conjugation by $P$.
\end{corollary}
\begin{proof}
In the following we adopt the notation that is used in the proof of Proposition \ref{prop:simple_gen}, namely $P$ is a standard parabolic subgroup with canonical simple system $\Delta_{P}=\lbrace \alpha_{1},\ldots,\alpha_{n-1} \rbrace\subseteq \Phi^{+}$ in sense of \cite[Section 4.1]{DL11}, $s_{\beta}=t$ for $\beta \in \Phi^{+}$ and $w\in P$ such that $\Delta=\Delta_{P}\cup \lbrace \gamma \rbrace$ with $\gamma=w(\beta)$ is a simple system of $\Phi$. Moreover, let $\Delta_{P}\cup \lbrace \alpha_{n} \rbrace\subseteq \Phi^{+}$ be the canonical simple system of $W$. By \cite[Proposition 5.7]{Hum90} we have that $\gamma=w(\beta)\in -\Phi^{+}$ if and only if $\ell_{S}(w s_{\beta})< \ell_{S}(w)$. The latter implies due to the strong exchange condition \cite[Theorem 5.8]{Hum90} that $t=s_{\beta}\in P$. But since $P$ is a proper parabolic subgroup with $\langle P,t \rangle=W$ we arrive at a contradiction. Hence $\gamma=w(\beta)\in \Phi^{+}$, and thus $\Delta\subseteq \Phi^{+}$. By \cite[Lemma 3 (b)]{DL11} $\Delta$ is the canonical simple system, i.e. we have $\Delta=\Delta_{P}\cup \lbrace \alpha_{n} \rbrace$. In particular, we found an element $w\in P$ with $wtw^{-1}=s_{\alpha_{n}}$.
\end{proof}

Now we prove Theorem \ref{thm:main2} which investigates uniformly the Hurwitz action on the set of reduced reflection factorizations of parabolic quasi-Coxeter elements in Weyl groups. It is already proven case-based in \cite{BGRW17} for finite Coxeter groups ae well as partially for simply laced Weyl groups in \cite{Voi85} and also case-based for affine Coxeter groups in \cite{Weg17a}.

\begin{proof}[Proof of Theorem \ref{thm:main2}]
We proceed by induction on the rank $n$. Let $w \in W$ be a quasi-Coxeter element and $(t_{1},\ldots,t_{n}),(t'_{1},\ldots,t'_{n})\in \Red_{T}(w)$ such that $\langle t_{1},\ldots,t_{n} \rangle =W$. We need to show that  $(t_{1},\ldots,t_{n})$ and $(t'_{1},\ldots,t'_{n})$ lie in the same Hurwitz orbit. If $n=1$ the statement is trivially satisfied. Thus assume that $n>1$. By Lemma \ref{lem:nice_factorization} we can assume that $P:=P(t_{1},\ldots,t_{n-1})\neq W$, and thus $P$ is a parabolic subgroup of rank $n-1$. Hence by Proposition \ref{prop:simple_gen} there exist $s_{1},\ldots,s_{n-1}\in P$ such that $\overline{S}:=\lbrace s_{1},\ldots,s_{n-1},t_{n} \rbrace$ is a simple system for $W$.
Let $x^{-1}=t_{1}\cdots t_{n-1}$ and consider the following path in the Bruhat graph in terms of the simple system $\overline{S}$
\[x\edge xt'_{1}\edge \ldots \edge xt'_{1}\cdots t'_{n}=t_{n}.\]
By Proposition \ref{prop:key_prop} there exists $\sigma\in B_{n}$ such that the factorization $(r_{1},\ldots, r_{n}):=\sigma(t'_{1},\ldots,t'_{n})$ is attached to the following path in the Bruhat graph
\[x\longleftarrow xr_{1}\longleftarrow \ldots \longleftarrow xr_{1}\cdots r_{i}\longrightarrow \ldots \longrightarrow xr_{1}\cdots r_{n}=t_{n}.\]
for $0\leq i \leq n$. If $i=n$, we have the decreasing path
\[x\longleftarrow xr_{1}\longleftarrow \ldots \longleftarrow xr_{1}\cdots r_{i}\longleftarrow \ldots \longleftarrow xr_{1}\cdots r_{n}=t_{n}.\]
Since $x\in P$ and $P$ is a proper standard parabolic subgroup, the strong exchange condition yields that $t_{n}=xr_{1}\cdots r_{n}\in P$ and thus we arrive to the contradiction
\[P=\langle P,t_{n} \rangle=\langle t_{1},\ldots,t_{n}\rangle=W.\]
Therefore we have $i<n$ and hence $1=\ell_{\overline{S}}(t_{n})>\ell_{\overline{S}}(xr_{1}\cdots r_{n-1})$. The latter implies directly $xr_{1}\cdots r_{n-1}=e$, which is equivalent to $r_{n}=t_{n}$. Therefore $(t'_{1},\ldots,t'_{n})$ and $(r_{1},\ldots,r_{n-1},t_{n})$ lie in the same Hurwitz orbit.

Now it suffices to show that $(r_{1},\ldots,r_{n-1})$ and $(t_{1},\ldots,t_{n-1})$ lie in the same Hurwitz orbit. By Theorem \ref{thm:fac_in_par_subg} we have that $r_{1},\ldots,r_{n-1}\in \langle t_{1},\ldots,t_{n-1}\rangle =P$, that is, $(r_{1},\ldots,r_{n-1})$ is a reduced reflection factorization of $t_{1}\cdots t_{n-1}$ in $P$. Hence $t_{1}\cdots t_{n-1}=r_{1}\cdots r_{n-1}$ is a parabolic quasi-Coxeter element by Theorem \ref{prop:prefix_quasi_cox}. Since $P$ is a (not necessary irreducible) Weyl group of rank $n-1$, the induction hypothesis yields that $(r_{1},\ldots,r_{n-1})$ and $(t_{1},\ldots,t_{n-1})$ lie in the same Hurwitz orbit. Altogether, $(t'_{1},\ldots,t'_{n})$ and $(t_{1},\ldots,t_{n})$ lie in the same Hurwitz orbit.

Now let $w$ be a parabolic quasi-Coxeter element such that $P(w)\neq W$. Then by Theorem \ref{thm:fac_in_par_subg} we have that $\Red_{T}(w)=\Red_{T\cap P(w)}(w)$. Thus we can restrict the investigation to the Weyl group $P(w)$, and hence $w$ is a quasi-Coxeter element of $P(w)$. By restricting to $P(w)$ we are in the situation that is already investigated previously.
\end{proof}

For finite Coxeter groups we also have the following converse statement.

\begin{proposition}[{\cite[Proposition 4.3]{BGRW17}}] \label{prop:bgrw_eine_richtung}
Let $(W,S)$ be a finite Coxeter system and $w\in W$ such that the Hurwitz action is transitive on its set of reduced reflection factorizations. Then $w$ is a parabolic quasi-Coxeter element.
\end{proposition}

\begin{corollary}
Let $W$ be a Weyl group and $w \in W$ a quasi-Coxeter
element. The factors of any reduced reflection factorization of w generate the group W.
\end{corollary}

Note that that Proposition \ref{prop:bgrw_eine_richtung} is proven uniformly in \cite{BGRW17}. Therefore the combination of this proposition and Theorem \ref{thm:main2} completes uniformly the picture for Weyl groups, that is, we obtain a uniform proof for the following result.

\begin{theorem}[{\cite[Theorem 1.1]{BGRW17}}] \label{thm:main2_weyl}
Let $W$ be a Weyl group and $w\in W$. Then the Hurwitz action is transitive on the set of reduced reflection factorizations of $w$ if and only if $w$ is a parabolic quasi-Coxeter element.
\end{theorem}

\begin{remark}
The example \cite[Example 5.7]{HK16} shows that in affine Coxeter groups there
exist elements with just one Hurwitz orbit which are not parabolic quasi-Coxeter elements. Therefore it is desirable to determine all those elements of affine Weyl groups with just one Hurwitz orbit.
\end{remark}

Based on Theorems \ref{thm:main1} and \ref{thm:main2} we uniformly obtain conditions on reflection factorization in Weyl groups to determine whether two reflection factorizations of an element lie in the same Hurwitz orbit.

\begin{theorem}[{\cite[Theorem 6.1]{LR16}}]\label{thm:hurw_quasi}
Let $W$ be a Weyl group and $w\in W$ a quasi-Coxeter element. Two reflection factorizations of $w$ lie in the same Hurwitz orbit if and only if they share the same multiset of conjugacy classes. 
\end{theorem}

Using the results of \cite{BGRW17} it remains valid for all finite Coxeter groups.

\begin{corollary}
Let $W$ be a Weyl group with simple system $S$. If the Coxeter graph of $(W,S)$ is connected and has a spanning tree with odd labels on all its edges, then two reflection factorizations of the same length of a quasi-Coxeter element in $W$ lie in the same Hurwitz orbit.
\end{corollary}

\section{Characterization of parabolic quasi-Coxeter elements in affine Coxeter groups}

As in Corollary \ref{cor:FinCharQuasi} for finite Coxeter groups, we aim to characterize parabolic quasi-Coxeter elements in affine Coxeter groups as well. Before we give another direct consequence of Theorem \ref{prop:prefix_quasi_cox} for affine Coxeter groups.

\begin{corollary} \label{cor:ParClosureFinite}
Let $(W,S)$ be an irreducible affine Coxeter system of rank $n$ and $t_{1},\ldots,t_{n}\in T$ with $\langle t_{1},\ldots,t_{n} \rangle=W$ and $P(t_1,\ldots, t_{n-1})\neq W$. Then $\langle t_{1},\ldots,t_{n-1} \rangle$ is a finite parabolic subgroup of rank $n-1$.
\end{corollary}

\begin{proof}
By Theorem \ref{prop:prefix_quasi_cox} we have that $t_{1}\cdots t_{n-1}$ is a proper parabolic quasi-Coxeter element, that is, $\langle t_{1},\ldots,t_{n-1} \rangle$ is a proper parabolic subgroup. Since all proper parabolic subgroups of an irreducible affine Coxeter system are finite, the claim follows.
\end{proof}

\begin{proposition} \label{prop:InfiniteAffSub}
Let $(W,S)$ be an irreducible affine Coxeter system of rank $n \geq 3$ with set of reflections $T$. Let $t_1, \ldots, t_n \in T$ such that $W= \langle t_1, \ldots, t_n \rangle$. If $\langle t_1, \ldots , t_{n-k} \rangle$ is finite for some $k$ with $2 \leq k <n$, then there exists $t \in \{ t_{n-k+1},\ldots,  t_n \}$ such that $\langle t_1, \ldots , t_{n-k}, t \rangle$ is finite as well.   
\end{proposition}

\begin{proof}
Write the reflection $t_i$ for $1 \leq i \leq n$ as $t_i = s_{\alpha_i, \ell_i}$ with $\alpha_i \in \Phi$ and $\ell_i \in \ZZ$.
First observe that, since $\langle t_1, \ldots, t_{n-k} \rangle$ is finite, the set of roots $\{ \alpha_1, \ldots, \alpha_{n-k} \}$ is linearly independent. For $n=3$ this is clear. To see this for $n >3$ let us assume, to the contrary, that $\{ \alpha_1, \ldots, \alpha_{n-k} \}$ is not linearly independent. Hence Carter's Lemma \ref{lem:Carter} implies that $s_{\alpha_1} \cdots s_{\alpha_{n-k}}$ is not reduced. By Theorem \ref{thm:main1} there exists a braid $\sigma \in \B_{n-k}$ such that
$$
\sigma (s_{\alpha_1}, \ldots, s_{\alpha_{n-k}}) = (s_{\beta_1}, \ldots , s_{\beta_{n-k-2}}, s_{\beta}, s_{\beta}).
$$ 
We can apply this braid to the factorization in the affine group as well. We obtain
$$
\sigma(t_1, \ldots , t_{n-k}) = (\star, \ldots, \star, s_{\beta,k_1}, s_{\beta, k_2})
$$
for integers $k_1, k_2 \in \ZZ$. If $k_1=k_2$, then the factorization $t_1 \cdots t_{n-k}$ is not reduced. If $k_1 \neq k_2$, then the infinite rank $2$ reflection subgroup $ \langle s_{\beta,k_1}, s_{\beta, k_2} \rangle$ is contained in the finite subgroup $\langle t_1, \ldots, t_{n-k} \rangle$. In both cases we arrive at a contradiction. Hence we have shown that $\{ \alpha_1, \ldots, \alpha_{n-k} \}$ is linearly independent. We distinguish two cases.

\textit{Case 1:} There exists $j \in \{n-k+1,\ldots, n\}$ such that the set of roots $\{ \alpha_1, \ldots, \alpha_{n-k}, \alpha_{j} \}$ is linearly independent. By \cite[Proposition 5.1]{Brady} the product $x:=t_1 \cdots t_{n-k}t_j$ is therefore elliptic, that is, $x$ is of finite order. By Lemma \ref{lem:Tits} the element $x$ is contained in a finite parabolic subgroup. In particular, the reflections $t_1, \ldots, t_{n-k},t_j$ are contained in this finite parabolic subgroup (see \cite[Section 1.4]{Bes03}). We conclude that
$$
\langle s_{\alpha_{1}, \ell_{1}}, \ldots, s_{\alpha_{n-k}, \ell_{n-k}}, s_{\alpha_{j}, \ell_{j}} \rangle = \langle t_1, \ldots , t_{n-k}, t_j \rangle
$$
is finite.

\textit{Case 2:} For all $j \in \{n-k+1,\ldots, n\}$ the set of roots $\{ \alpha_1, \ldots, \alpha_{n-k}, \alpha_{j} \}$ is not linearly independent. But then 
$$\dim_{\RR} \text{span}_{\RR}(\alpha_1, \ldots , \alpha_{n})=n-k \leq n-2;$$ 
a contradiction, because $\langle s_{\alpha_1}, \ldots , s_{\alpha_n} \rangle$ is a finite Coxeter group of rank $n-1$ (since $\langle t_1,\ldots, t_n\rangle$ is irreducible and affine of rank $n$). Hence this case does not occur.
\end{proof}

\begin{corollary} \label{cor:IndStep}
Let $(W,S)$ be an irreducible affine Coxeter system of rank $n \geq 3$ with set of reflections $T$. Let $t_1, \ldots, t_n \in T$ such that $W= \langle t_1, \ldots, t_n \rangle$. If $\langle t_1, \ldots , t_{n-k} \rangle$ is finite for some $k$ with $2 \leq k <n$, then there exist reflections $r_{n-k+1}, \ldots , r_n$ with
$$
(t_{n-k+1}, \ldots , t_n) \sim (r_{n-k+1}, \ldots , r_n)
$$ 
such that $\langle t_1, \ldots , t_{n-k}, r_{n-k+1}, \ldots , r_{n-1} \rangle$ is finite.   
\end{corollary}

\begin{proof}
By Proposition \ref{prop:InfiniteAffSub} there exists $i \in \{ 1, \ldots, k\}$ such that $\langle t_1, \ldots, t_{n-k}, t_{n-k+i} \rangle$ is finite. Applying the Hurwitz action, we obtain
$$
(t_1, \ldots, t_n) \sim (t_1, \ldots, t_{n-k}, t_{n-k+i}, t_{n-k+i}t_{n-k+1}t_{n-k+i}, \ldots , t_{n-k+i} t_{n-k+i-1} t_{n-k+i}, t_{n-k+i+1}, \ldots , t_n ).
$$
Note that 
$$
W = \langle t_1, \ldots, t_{n-k}, t_{n-k+i}, t_{n-k+i} t_{n-k+1} t_{n-k+i}, \ldots , t_{n-k+i} t_{n-k+i-1} t_{n-k+i}, t_{n-k+i+1}, \ldots , t_n \rangle.
$$
Since $\langle t_1, \ldots, t_{n-k}, t_{n-k+i} \rangle$ is finite and if $k-1\geq 2$, we can apply Proposition \ref{prop:InfiniteAffSub} as before. Proceeding in this manner, we eventually obtain after $k-1$ steps the claimed finite reflection subgroup.
\end{proof}

We are now able to prove Corollary \ref{main4}.

\begin{proof}[Proof of Corollary \ref{main4}]
We only have to prove the \textit{if} direction. If $\ell_T(x)=n-1$, this is precisely Theorem \ref{prop:prefix_quasi_cox}. Let us therefore assume that $\ell_T(x)=n-k$ for some $k$ with $2 \leq k < n$. Let $w$ be a quasi-Coxeter element and $(t_1, \ldots, t_n) \in \Red_T(w)$ such that $x = t_1 \cdots t_{n-k}$. By \cite[Theorem 1.1]{Weg17a} we have $W= \langle t_1, \ldots, t_n \rangle$. Hence we can apply Corollary \ref{cor:IndStep} to find reflections $r_{n-k+1}, \ldots , r_n$ with
$$
(t_1, \ldots, t_{n-k}, t_{n-k+1}, \ldots , t_n) \sim (t_1, \ldots, t_{n-k}, r_{n-k+1}, \ldots , r_n)
$$
such that $P:=\langle t_1, \ldots, t_{n-k}, r_{n-k+1}, \ldots , r_{n-1} \rangle$ is finite. By Lemma \ref{lem:Tits} the parabolic closure 
$$P(t_1, \ldots, t_{n-k}, r_{n-k+1}, \ldots , r_{n-1})$$ 
is finite as well. In particular, we have $P(t_1, \ldots, t_{n-k}, r_{n-k+1}, \ldots , r_{n-1}) \neq W$. Therefore $P$ is parabolic by Corollary \ref{cor:ParClosureFinite}. Since 
$$
x = t_1 \cdots t_{n-k} \leq_T t_1 \cdots t_{n-k} r_{n-k+1} \cdots r_{n-1},
$$
the element $x$ is a parabolic quasi-Coxeter element in $P$ by Corollary \ref{cor:FinCharQuasi}. As $P$ is parabolic in $W$, the element $x$ is a parabolic quasi-Coxeter element in $W$ as well.

The remaining assertion is now a direct consequence of Corollary \ref{cor:FinOrdQuasi}.
\end{proof}

\begin{remark}
We make use of \cite[Theorem 1.1]{Weg17a} in the proof of Corollary \ref{main4}. The proof given in \cite{Weg17a} include results of \cite{BGRW17} and \cite{LR16}, both of which are proved by a case-based analysis. But for both results, namely Theorem \ref{thm:main1} and Theorem \ref{thm:main2}, we provide uniform proofs in this paper. In particular, this provides a uniform of \cite[Theorem 1.1]{Weg17a}, making all proofs in this paper uniform.
\end{remark}

\bibliographystyle{amsplain}
\bibliography{QuasiCoxeter}
\end{document}